\documentclass[letterpaper]{article}

\usepackage[T1]{fontenc}
\usepackage{setspace}
\usepackage[letterpaper,margin=1in]{geometry}
\usepackage{graphicx} 
\usepackage{amsmath}
\usepackage{amsthm}
\usepackage{amssymb}
\usepackage{booktabs}
\usepackage{longtable}
\usepackage{multirow}
\usepackage{mathrsfs}
\usepackage{makecell}
\usepackage{caption}
\usepackage[backend=biber,
style=chem-acs,
articletitle=true,
sorting=none
]{biblatex} 
\usepackage{hyperref}
\usepackage{authblk}
\usepackage{subcaption}

\usepackage{placeins}
\usepackage{float}

\usepackage{relsize}              
\newcommand{\bigveebar}{%
  \mathop{\mathlarger{\mathlarger{\veebar}}}\displaylimits%
}

\newtheorem{theorem}{Theorem}[section]
\newtheorem{corollary}{Corollary}[section]

\usepackage{tikz}
\definecolor{lime}{HTML}{A6CE39}
\DeclareRobustCommand{\orcidicon}{
	\begin{tikzpicture}
	\draw[lime, fill=lime] (0,0) 
	circle [radius=0.16] 
	node[white] {{\fontfamily{qag}\selectfont \tiny ID}};
	\draw[white, fill=white] (-0.0625,0.095) 
	circle [radius=0.007];
	\end{tikzpicture}
	\hspace{-2mm}
}
\foreach \x in {A, ..., Z}{\expandafter\xdef\csname orcid\x\endcsname{\noexpand\href{https://orcid.org/\csname orcidauthor\x\endcsname}
			{\noexpand\orcidicon}}
}

\DeclareCiteCommand{\cite}[\mkbibsuperscript]
  {\usebibmacro{prenote}}
  {%
    \bibhyperref{\printfield{labelnumber}}%
  }
  {\multicitedelim}
  {\usebibmacro{postnote}}

\DeclareCiteCommand{\parencite}[\mkbibsuperscript]
  {\usebibmacro{prenote}}
  {%
    \bibhyperref{\printfield{labelnumber}}%
  }
  {\multicitedelim}
  {\usebibmacro{postnote}}

\DeclareCiteCommand{\textcite}
  {\usebibmacro{prenote}}
  {%
    \printnames{labelname}%
    \mkbibsuperscript{\bibhyperref{\printfield{labelnumber}}}%
  }
  {\multicitedelim}
  {\usebibmacro{postnote}}

\title{Mixed-Integer Reaggregated Hull Reformulation of Special Structured Generalized Linear Disjunctive Programs}

\author[1]{Albert Joon Lee\orcidB}
\author[1,*]{David E. Bernal Neira\orcidA}

\affil[1]{Davidson School of Chemical Engineering, Purdue University, West Lafayette, IN 47907, USA}

\date{*Email: dbernaln@purdue.edu}

\date{}

\begin{document}

\maketitle

\begin{abstract}
    Generalized Disjunctive Programming (GDP) provides a powerful framework for combining algebraic constraints with logical disjunctions. To solve these problems, mixed-integer reformulations are required, but traditional reformulation schemes, such as Big-M and Hull, either yield a weak continuous relaxation or result in a bloated model size.
    \textcite{castro2012generalized} showed that scheduling problems can be formulated as GDP by modeling task orderings as disjunctions with algebraic timing constraints.
    Moreover, in their work, a particular representation of the single-unit scheduling problem, namely using a time-slot concept, can be reformulated as a tight yet compact mixed-integer linear program with notable computational performance. 
    Based on that observation, and focusing on the case where the constraints in disjunctions are linear and share the same coefficients, we connect the characterization of the convex hull of these disjunctive sets by \textcite{jeroslow1988simplification} and \textcite{blair1990representation} with Castro and Grossmann’s time-slot reaggregation strategy to derive a unified reformulation methodology.
    We test this reformulation in two problems, single-unit scheduling and two-dimensional strip-packing (also worked by \textcite{castro2012time}).
    We derive new formulations of the general precedence concept of single-unit scheduling and symmetry-breaking formulations of the strip-packing problem, yielding mixed-integer programs with strong theoretical guarantees, particularly compact formulations in terms of continuous variables, and efficient computational performance when solving them with commercial mixed-integer solvers for these problems.
\end{abstract}


\section{Introduction}
In this section, we introduce the Generalized Disjunctive Programming (GDP) framework, review the GDP formulations for the single‐unit scheduling problem presented by \textcite{castro2012generalized}, and summarize our contributions to developing and evaluating a fully reaggregated hull reformulation for special structured generalized linear disjunctive programs.

\subsection{Generalized Disjunctive Programming}

Generalized Disjunctive Programming (GDP) is a mathematical optimization framework that involves both discrete and continuous variables, as well as algebraic constraints and logical constraints, including disjunctions.
The disjunctive terms in a GDP model represent alternative system configurations or operational modes, where each disjunct is associated with specific sets of constraints, which the corresponding Boolean variables can selectively activate.
In this work, we focus on the case where the constraints within the disjunctions are linear inequalities, hence resulting in Generalized Linear Disjunctive Programs (GLDP); see, for example, \textcite{sawaya2012hierarchy},

\begin{align*}
\tag{GDP}
\label{prob:GDP}
\min_{\mathbf{x},\mathbf{Y}} \quad & f(\mathbf{x})  \\
\text{s.t.} \quad & g(\mathbf{x}) \leq \mathbf{0} \\
& \bigveebar_{j \in D_k} \begin{bmatrix}  Y_{jk} \\ \mathbf{A}_{jk}\mathbf{x} \leq \mathbf{b}_{jk} \end{bmatrix}, \quad \forall k \in K \\
& \Omega(\mathbf{Y}) = \text{True} \\
& \mathbf{x} \in [\mathbf{x}^L, \mathbf{x}^U] \subseteq \mathbb{R}^{n} \\
& Y_{jk} \in \{\text{False}, \text{True}\}, \quad \forall j \in D_k,\ \forall k \in K. \\
\end{align*}
Here, the optimization is performed over the \(n\)-dimensional continuous variable \(\mathbf{x}\) bounded by vectors \((\mathbf{x}^L, \mathbf{x}^U) \in \mathbb{R}^n\) and \(n_Y\)=dimensional Boolean variable \(\mathbf{Y}\).
The objective function \(f: \mathbb{R}^n \to \mathbb{R}\) can be represented as \(\mathbf{c}^{\top}\mathbf{x}\) if linear, where \(\mathbf{c}\in\mathbb{R}^n\) is the objective coefficient vector.
The global constraints \(g: \mathbb{R}^n \to \mathbb{R}^m\) represent constraints that need to be satisfied regardless of the disjunction values.
If \(g(\mathbf{x})\) is a set of linear functions, it can be represented as  \(\mathbf{B}\mathbf{x} \leq \mathbf{b}\), where \(\mathbf{B}\in\mathbb{R}^{m\times n}\) and \(\mathbf{b}\in\mathbb{R}^m\).
Each disjunct \(j\in D_k\) enforces linear constraints of the form \(\mathbf{A}_{jk}\mathbf{x}\le \mathbf{b}_{jk}\) with \(\mathbf{A}_{jk}\in\mathbb{R}^{m_{jk}\times n}\) and $\mathbf{b}_{jk}\in\mathbb{R}^{m_{jk}}$ in case its corresponding Boolean variable \(Y_{jk}\) is set to \(\text{True}\).
The exclusive-or constraint, denoted by the operator \(\veebar\), ensures exactly one disjunct \(j\in D_k\) is active for each disjunction \(k \in K\).
The logical proposition \(\Omega: \{\text{False}, \text{True}\}^{n_{Y}} \to \{\text{False}, \text{True}\}\) enforces the activation of disjuncts and any additional logical relations among Boolean variables.
Note that \(n_Y\) refers to the number of Boolean variables in the GDP.
In what follows, we will also focus on disjunctions whose disjuncts have linear constraints that share a common coefficient matrix, i.e., \(\mathbf{A}_{jk}\equiv \mathbf{A}_k\) for all \(j\in D_k\), a structure we later exploit in our reaggregated hull reformulation (RHR).

\subsection{Mixed-integer Programming Reformulations of GDP}
\label{sec:mip_reformulations}

A common approach to solving GDPs is to reformulate them into a mixed-integer program (MIP) that standard solvers can handle \parencite{grossmann2003generalized, trespalacios2014review, sawaya2012hierarchy}. 
The core strategy of these reformulations is to introduce binary variables \(y_{jk} \in \{0, 1\}\), which encode the activation of each disjunct \parencite{vielma2015mixed}.
The general structure of such MIP reformulations can be summarized as

\begin{align*}
\tag{MIP}
\label{prob:mip}
\min_{\mathbf{x}, \mathbf{y}} \quad & f(\mathbf{x}) \\
\text{s.t.} \quad &  (\mathbf{x},\mathbf{y}) \in R \\
 & g(\mathbf{x}) \leq \mathbf{0} \\
 & \mathbf{D}\mathbf{y} \leq \mathbf{d} \\
    & \mathbf{x} \in [\mathbf{x}^L, \mathbf{x}^U ] \subseteq \mathbb{R}^{n}\\
    & y_{jk} \in \{0, 1 \}, \quad \forall j \in D_k,\ \forall k \in K. \\
\end{align*}

In this problem, \(\mathbf{y}\) is the binary vector obtained by reformulating the Boolean variables \(\mathbf{Y}\), and the constraints 
\(\mathbf{D}\mathbf{y} \leq \mathbf{d}\) represent the logical propositions originally encoded as \((\Omega({Y}_{jk}) = \text{True})\) from problem \eqref{prob:GDP}.
The problem \eqref{prob:mip} provides no details on how the reformulation is encoded as a mixed-integer program; it does not yet include the explicit algebraic reformulation of the disjunctions that we denote as the set \(R\). 
To complete the reformulation, each disjunction must be expressed using algebraic constraints that enforce the conditional logic of the model.
Various reformulations exist for this step, of which we will cover the Big-M and hull reformulations, and contribute the reaggregated hull reformulation for the special case where \(\mathbf{A}_{jk}\equiv \mathbf{A}_k \quad \forall j\in D_k\).

One way to represent the disjunctions \(k \in K\) into mixed-integer form is to introduce Big-M constraints \cite{raman1994modelling}.
Here, constraints are enforced exactly or trivially by having a large coefficient relax them when their corresponding binary variable is equal to zero, which yields
\begin{align*}
\tag{BM}
\label{eq:BigM}
R =  \left\{ (\mathbf{x},\mathbf{y}) : 
\begin{aligned}
 & \mathbf{A}_{jk}\mathbf{x} - \mathbf{b}_{jk} \leq \mathbf{M}_{jk} (1 - y_{jk}) \quad \forall k \in K, \ j \in D_k, \\
 & \sum_{j \in D_k} y_{jk} = 1 \quad \forall k \in K
 \end{aligned}
 \right\}.
\end{align*}
Since each disjunctive constraint \(\mathbf{A}_{jk}\mathbf{x} \leq \mathbf{b}_{jk}\) is linear and the bounds of the variables \(\mathbf{x}\) are given as \(\mathbf{x}^L \leq \mathbf{x} \leq \mathbf{x}^U\), the valid Big-M coefficient \(\mathbf{M}_{jk}\) can be bounded by computing the maximum violation of the linear constraints in the disjunctions $\mathbf{M}_{jk} \geq \max_{\mathbf{x} \in [\mathbf{x}^L,  \mathbf{x}^U]}\{\mathbf{A}_{jk}\mathbf{x} -\mathbf{b}_{jk}\}$.

The Big-M reformulation can be advantageous because of its simplicity and low modeling overhead; however, if the chosen \(M\) constant is too small, it might fail to enforce the logical condition in the original \eqref{prob:GDP}.
If it is large, the relaxation becomes weak and might even lead to numerical instability, which can increase the number of iterations required for convergence~\cite{vielma2015mixed}.

On the other hand, the hull reformulation (HR) disaggregates the vector of continuous variables \(\mathbf{x}\) for each disjunction \(k \in K\) with the index set \(D_k\) into disaggregated variables \(\hat{\mathbf{x}}_{jk}\) for \(j \in D_k\) and these disaggregated variables \(\hat{\mathbf{x}}_{jk}\) are associated with binary variables \(y_{jk}\) \cite{grossmann2003generalized}.
This yields
\begin{align*}
\tag{HR}
\label{eq:Hull}
R =  \left\{ (\mathbf{x},\mathbf{y}) : 
\begin{aligned}
    & \mathbf{x} = \sum_{j \in D_k} \hat{\mathbf{x}}_{jk} \quad \forall k \in K, \\
    & \mathbf{A}_{jk}\,\hat{\mathbf{x}}_{jk}- \mathbf{b}_{jk}\,y_{jk} \leq 0 \quad \forall k \in K, \ j \in D_k, \\
    & \mathbf{x}^L y_{jk} \leq \hat{\mathbf{x}}_{jk} \leq \mathbf{x}^U y_{jk} \quad \forall k \in K, \ j \in D_k, \\
    & \sum_{j \in D_k} y_{jk} = 1 \quad \forall k \in K, \\
    & \hat{\mathbf{x}}_{jk} \in [\mathbf{x}^L, \mathbf{x}^U ] \subseteq \mathbb{R}^{n} \quad  \forall k \in K, \ j \in D_k\\
\end{aligned}
\right\}.
\end{align*}
As the disaggregated variables \(\hat{\mathbf{x}}_{jk}\) and the binary variables \(y_{jk}\) are coupled, \(\hat{\mathbf{x}}_{jk} = 0\) when the corresponding disjunct is inactive, i.e., \((y_{jk}=0\)).
In contrast, when the disjunct is selected \((y_{jk}=1)\), the disaggregated variable \(\hat{\mathbf{x}}_{jk}\) is allowed to take values within the same bounds as the original variable \(\mathbf{x}\), that is, \([\mathbf{x}^L, \mathbf{x}^U]\).
By assigning the same lower and upper limits as the original variable to each disaggregated variable, the formulation ensures that active disjuncts respect their specific feasible ranges and inactive disjuncts do not affect the solution.
The HR constructs the convex hull of each disjunctive set, as shown in Theorem \ref{th:chull}, yielding a tighter relaxation than the Big-M reformulation, potentially reducing the number of nodes required to solve the problem to in a branch-and-bound algorithm \parencite{jeroslow1988alternative,balas1988convex,vielma2015mixed}.
However, because it disaggregates each original variable into one variable per disjunct, the model size grows substantially, making the problem potentially more expensive to solve for larger models \parencite{vielma2015mixed,balas1988convex}.

Note that having the same left-hand coefficients in the linear constraints of the disjunctions would only mean replacing the coefficients \(\mathbf{A}_{jk}\equiv \mathbf{A}_k \quad \forall j\in D_k\) for either reformulation.

Notice that in the case where the objective function and the global constraints are linear, i.e. \(f(\mathbf{x}) = \mathbf{c}^{\top}\mathbf{x}\), \(g(\mathbf{x})=\mathbf{B}\mathbf{x}- \mathbf{b}\), both the \ref{eq:BigM} and the \ref{eq:Hull} reformulations yield a problem \eqref{prob:mip} that is a mixed-integer linear program (MILP).
This valuable property is also inherited by the proposed RHR method, as the reformulation of equal left-sided linearly constrained disjunctions yields a mixed-integer linear set, as described below.

In addition to the Big-M and hull reformulations, alternative techniques have been developed to accelerate solution time.
For example, \textcite{trespalacios2015algorithmic} used a hybrid reformulation of GDP that combines the strengths of Big-M and hull reformulations.
Multi-parameter Big-M (MBM) strategies assign distinct \(M\)-values to individual disjunctive constraints, yielding a strictly tighter relaxation than the classical Big-M formulation without increasing the number of binary variables \cite{trespalacios2015improved}.
Finally, the $p$-split formulation constructs a hierarchy between Big-M and hull reformulations by dividing each disjunction into $p$ partitions and applying hull reformulation within each partition, so that $p=1$ yields Big-M and $p$ equal to the number of terms recovers the hull reformulation over the whole disjunction \cite{kronqvist2025p}.



\subsection{Single-Unit Scheduling as a GDP}
The single-unit scheduling problem is defined as follows.
The problem is to minimize the overall completion time (make-span) of jobs with given processing times, release dates, and due dates on a single machine.
Suppose that each instance involves a set \(I\) of jobs, where each job \(i\) has a processing time \(p_i\), a release time \(r_i\), due time \(d_i\), and start time \(x_i\).
There were no changeovers, and the objective is to minimize the overall make-span \(MS\).
\textcite{castro2012generalized} formulated the single-unit sequencing problem as a Generalized Disjunctive Programs (GDP) using three different concepts: general precedence, immediate precedence, and time-slot.
When these GDPs were reformulated into MIP as described in \S \ref{sec:mip_reformulations}, various MIP formulations in the literature were rediscovered, besides proposing new models.
By formulating the time slot concept, which uses symmetry-breaking concepts, as a GDP, a more compact and efficient MIP formulation was discovered \cite{castro2012generalized}.

\subsubsection{General Precedence Concept}

The general precedence (GP) concept models single-machine sequencing by setting a binary precedence decision on every pair of jobs.
Here, \(x_i\in\mathbb{R}_+\) denotes the start time of job \(i\), and \(Y_{ij}\in\{\mathrm{False},\mathrm{True}\}\) indicates whether job \(i\) precedes job \(j\).
Given every pair \((i,j)\) with \(i<j\), a Boolean variable \(Y_{ij}\) selects one of two exclusive constraints. 
If \(Y_{ij}\) is true, job \(i\) must be completed before the job \(j\) starts \((x_i + p_i \leq x_j)\). 
Otherwise, the job \(j\) precedes the job \(i\) \((x_j + p_j \leq x_i)\).
In addition to the disjunctions, each job must start no earlier than its release time
\begin{equation}
\label{eq:gp_lb}
r_i \leq x_i \quad \forall i \in I.
\end{equation}
It must also be completed no later than its due time
\begin{equation}
\label{eq:gp_ub}
x_i + p_i \leq d_i \quad \forall i \in I.
\end{equation}
Each completion time \(x_i + p_i\) is also bounded by the make-span through \(x_i + p_i \leq MS\), allowing the total make-span to be modeled by minimizing the right-hand side of these constraints.
The model is given as

\begin{align*}
\tag{GP}
\label{prob:gdp_gp}
\min_{MS, \mathbf{x}, \mathbf{Y}} &\ MS \\
\text{s.t.} &\ r_i \leq x_i \quad \forall i \in I \\
&\ x_i + p_i \leq d_i \quad \forall i \in I  \\
&\ \begin{bmatrix} 
Y_{ij} \\
x_i + p_i \leq x_j 
\end{bmatrix}
\veebar
\begin{bmatrix} 
\neg Y_{ij} \\
x_j + p_j \leq x_i 
\end{bmatrix} \quad \forall i \in I, \ \forall j \in I, \ i<j  \\
&\ x_i + p_i \leq MS \quad \forall i \in I \\
    & Y_{ij} \in \{\text{False},\text{True}\} \quad \forall i \in I,\ \forall j \in I \\
    & x_i \in \mathbb{R}_+ \quad \forall i \in I \\
    & MS \in \mathbb{R}_+.
\end{align*}

Figure \ref{fig:GP-Single} illustrates the general precedence concept by showing how the Boolean variables \(Y_{ij}\) enforce a sequencing relationship between all pairs of jobs \((i,j)\) which \(i<j\), selecting whether job \(i\) precedes job \(j\) or vice versa.

\begin{figure}
    \centering
    \includegraphics[width=0.7\linewidth]{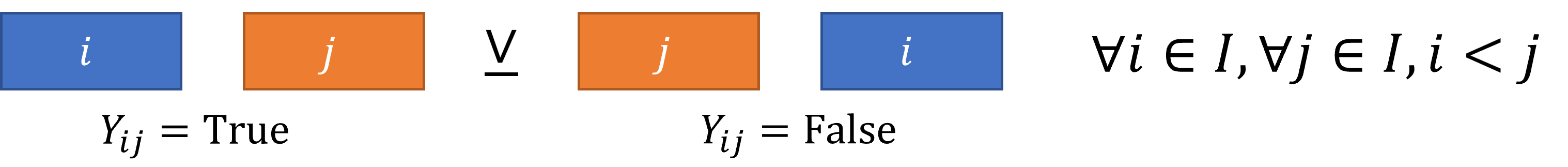}
    \caption{General precedence for single unit problem. Each Boolean variable determines whether job \(i\) is scheduled before job \(j\), or job \(j\) before job \(i\).}
    \label{fig:GP-Single}
\end{figure}

\subsubsection{Immediate Precedence Concept}

The immediate precedence (IP) concept of the scheduling problem models the single-unit sequence as a set of Boolean disjunctions that spot each job's neighbors in the processing order.
For every job \(i\), the Boolean variables \(Y_{ij} \ (j\neq i)\) can activate exactly one successor.
Here, \(x_i\in\mathbb R_+\) is the start time of job \(i\), \(Y_{ij}\in\{\text{False},\text{True}\}\) indicates that \(i\) immediately precedes \(j\), and \(Y_i^{\text{first}},Y_i^{\text{last}}\in\{\text{False},\text{True}\}\) mark \(i\) as the first or last job, respectively.
If one of these variables is true, the corresponding job \(j\) must start immediately after job \(i\) finishes; otherwise, the Boolean variable \(Y_i^{\text{last}}\) is set and \(i\) becomes the last job.
For the mirror sets of Boolean variables, \(Y_{ji}\) and \(Y_{i}^\text{first}\), assign a unique immediate predecessor or tag \(i\) as the first job.
Disjunctions ensure exactly one first and one last job, and logical constraints prevent any job from being both first and last.
Just like the general precedence concept, immediate precedence has the feasibility bounds for the starting time and the make-span linkage.
The notion of immediate precedence of the model is given as

\begin{align*}
\label{eq:gdp_ip}
\tag{IP}
\min_{MS, \mathbf{x}, \mathbf{Y}} &\ MS \\
\text{s.t.} &\ r_i \leq x_i \quad \forall i \in I  \\
&\ x_i + p_i \leq d_i \quad \forall i \in I  \\
&\ x_i + p_i \leq MS \quad \forall i \in I \\
&\ \veebar_{j \neq i} \begin{bmatrix}
    Y_{ij} \\
    x_i +p_i \leq x_j
\end{bmatrix} \veebar \begin{bmatrix}
    Y_{i}^{\text{last}} \\
    x_j + p_j \leq x_i \ \forall j \neq i
\end{bmatrix} \quad \forall i \in I \\
&\ \veebar_{j \neq i} \begin{bmatrix}
    Y_{ji} \\
    x_j +p_j \leq x_i
\end{bmatrix} \veebar \begin{bmatrix}
    Y_{i}^{\text{first}} \\
    x_i + p_i \leq x_j \ \forall j \neq i
\end{bmatrix} \quad \forall i \in I \\
&\ \veebar_{i \in I} \begin{bmatrix}
    Y_{i}^{\text{first}} \\
    x_i + p_i \leq x_j \ \forall j \neq i
\end{bmatrix} \\
&\ \veebar_{i \in I} \begin{bmatrix}
    Y_{i}^{\text{last}} \\
    x_j + p_j \leq x_i \ \forall j \neq i
\end{bmatrix} \\
&\ \neg \left( Y_{i}^{\text{first}} \land Y_{i}^{\text{last}} \right) \quad \forall i \in I\\
    & Y_{ij} \in \{\text{False},\text{True}\} \quad \forall i,j \in I, i \neq j \\
    & Y_{i}^{k} \in \{\text{False},\text{True}\} \quad \forall i \in I, k \in \{\text{first},\text{last}\} \\
    & x_i \in \mathbb{R}_+ \quad \forall i \in I \\
    & MS \in \mathbb{R}_+.
\end{align*}

\begin{figure}
    \centering
    \includegraphics[width=0.7\linewidth]{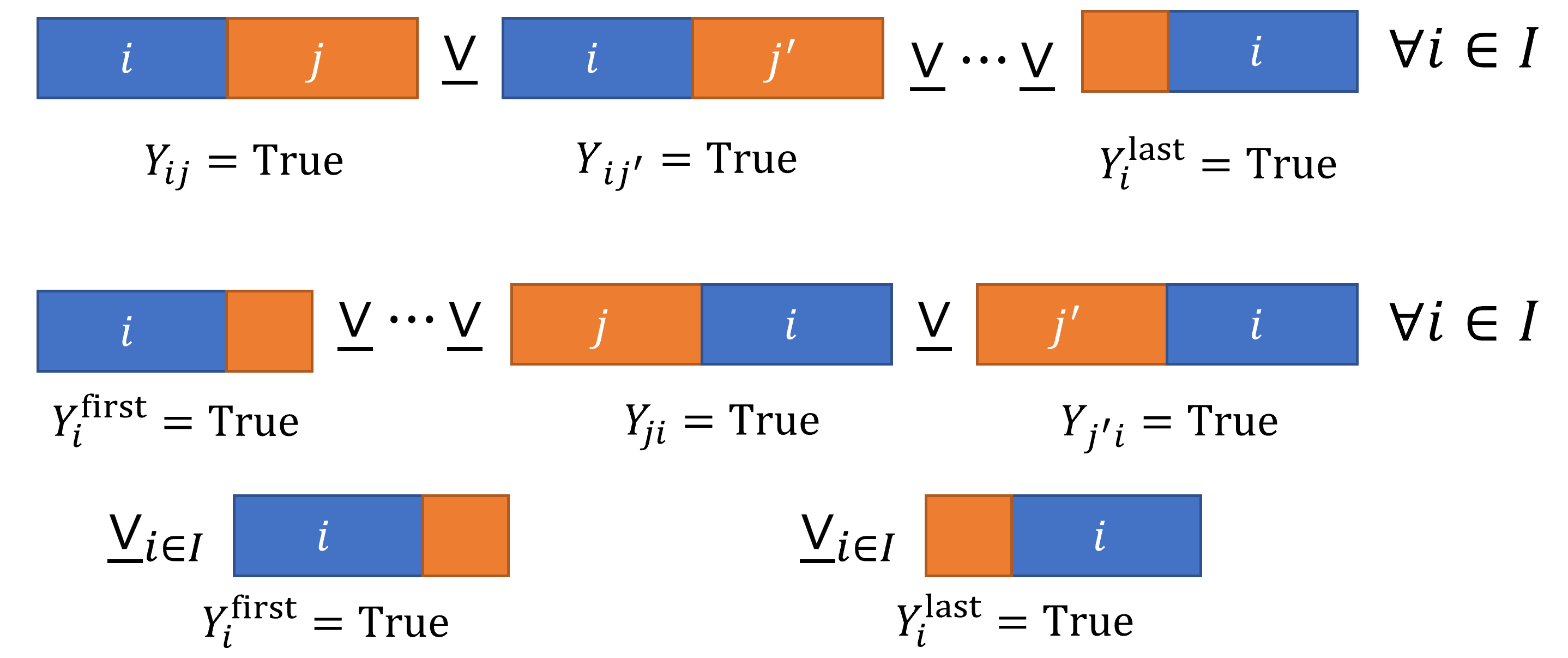}
    \caption{Immediate precedence concept for single unit problem. Disjunctions assign each job a relative position by enforcing pairwise orderings and selecting exactly one first and one last job.}
    \label{fig:IP-Single}
\end{figure}

Figure~\ref{fig:IP-Single} illustrates the concept of immediate precedence, where each disjunct determines whether job \(i\) precedes a specific job \(j\) or is the last job (using \(Y_{ij}\) and \(Y_i^\text{last}\)), or whether job \(i\) follows all other jobs or is the first job (using \(Y_{ji}\) and \(Y_i^\text{first}\)).
Additional disjunctions ensure that exactly one job is assigned as the first and one as the last, and these two roles must be assigned to different jobs.

\subsubsection{Time-Slot Concept}

The time slots (TS) concept defines a grid consisting of \(\lvert T \rvert = \lvert I \rvert \) time slots \(\lvert T \rvert\) and assigns order to the slots.
A Boolean variable \(Y_{it}\) is true precisely when job \(i\) is assigned at slot \(t\); otherwise, it is false.
For each job \(i\), exactly one \(Y_{it}\) is true, ensuring one job is assigned for a single time slot, which yields the following GDP formulation

\begin{align*}
\label{prob:gdp_ts}
\tag{TS}
   \min_{MS, \mathbf{x}, \mathbf{Y}}  &\ MS \\
\text{s.t.} &\ \veebar_{i} \begin{bmatrix}
    Y_{it} \\
    MS\rvert_{t = \lvert T \rvert} + x_{t+1}\rvert_{t \neq \lvert T \rvert} - x_t \geq p_i \\
    x_t \geq r_i \\
    x_t + p_i \leq d_i
\end{bmatrix} \quad \forall t \in T  \\ 
&\ \veebar_t Y_{it} \quad \forall i \in I \\
    & Y_{it} \in \{\text{False},\text{True}\} \quad \forall i \in I,\ \forall t \in T \\
    & x_t \in \mathbb{R}_+ \quad \forall t \in T \\
    & MS \in \mathbb{R}_+,
\end{align*}
where \(x_t \in \mathbb{R}_+\) denotes the start time of slot \(t\).
Whenever slot \(t\) assigned to job \(i\), i.e., \(Y_{it}=\text{True}\), slot \(t\) cannot start before job \(i\)’s release \((x_t \geq r_i)\), must finish by its deadline\(x_t + p_i \leq d_i\), and reserves at least \(p_i\) time until the next event (or make-span) \((MS\rvert_{t = \lvert T \rvert} + x_{t+1}\rvert_{t \neq \lvert T \rvert} - x_t \geq p_i )\).
The time-slot concept replaces the pairwise sequencing decisions with an assignment of each task to a time slot.
Each slot has a Boolean variable that assigns exactly one job to each slot and one slot to each job, thereby implicitly ordering tasks by slot index.
Therefore, the time slot concept yields a significantly smaller MILP formulation than precedence-based models, potentially resulting in faster solution times.

\begin{figure}
    \centering
    \includegraphics[width=0.7\linewidth]{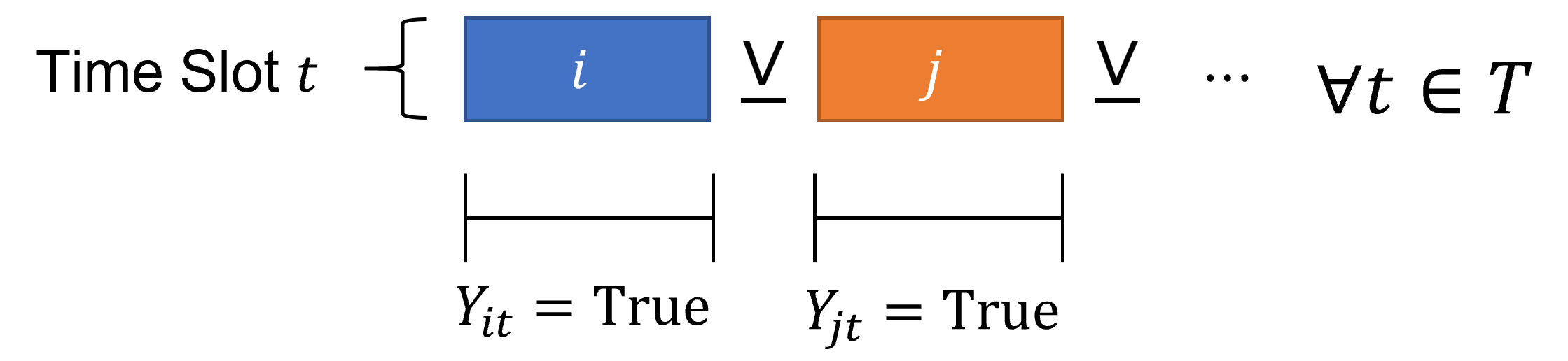}
    \caption{Time slot concept for single unit problem. Each job is assigned to exactly one time slot, establishing an implicit job sequence based on slot order.}
    \label{fig:TS-Single}
\end{figure}

Figure~\ref{fig:TS-Single} illustrates the time-slot concept, where a fixed grid of \(\lvert T \rvert = \lvert I \rvert\) time slots is defined.
Each disjunction enforces that exactly one job \(i\) is assigned to a specific slot \(t\) using the Boolean variable \(Y_{it}\).

\subsubsection{Mixed-integer programming formulations from GDP models}

In \textcite{castro2012generalized}, the \eqref{eq:BigM} and \eqref{eq:Hull} reformulations presented in \S \ref{sec:mip_reformulations}.
For the time-slot problem \eqref{prob:gdp_ts}, \textcite{castro2012generalized} proposed a particular reformulation that aggregates the disaggregated variables of the hull reformulation and yields an MILP of the form,
\begin{align*}
\tag{TS\_RHR}
\label{eq:TSCHC}
        \min_{ MS, \mathbf{x},\mathbf{y}} \quad & MS \\
    \text{s.t.} \quad 
    & MS\rvert_{t=\lvert T \rvert} + x_{t+1}\rvert_{t\neq \lvert T \rvert} - x_t \geq \sum_{i \in I} p_i y_{it} && \forall t \in T \\
    & x_t \geq \sum_{i\in I} r_i y_{it} && \forall t \in T \\
    & x_t \leq \sum_{i\in I}(d_i - p_i) y_{it} && \forall t \in T \\
    & \sum_{t \in T} y_{it} = 1 && \forall i \in I \\
    & \sum_{i \in I} y_{it} = 1 && \forall t \in T \\
    & y_{it} \in \{0,1\} && \forall i \in I,\ \forall t \in T \\
    & x_t \in \mathbb{R}_+ && \forall t \in T \\
    & MS \in \mathbb{R}_+.
\end{align*}

Equation~\eqref{eq:TSCHC} defines a compact time-slot reformulation in which \(x_t\) denotes the start time of slot \(t\), and \(y_{it}\) is a binary variable indicating whether job \(i\) is assigned to slot \(t\). 
Recognizing that the constraints in the \eqref{eq:Hull} reformulation for model \eqref{prob:gdp_ts} share identical left-hand sides, the model replaces them with a single aggregated inequality per slot. 
This eliminates the need for multiple continuous variables and linking constraints, while preserving HR tightness. 
As a result, the model achieves substantial reductions in size and enhanced computational performance \cite{castro2012generalized}.

\subsubsection{Applications of the Time-Slot Concept}

Beyond theoretical advances, the time-slot GDP framework has proven its practical value in diverse scheduling applications.
For example, \textcite{castro2015operating} applied the slot-based model to operating room planning by constructing a time grid for each combination of room and day, and assigning cases (surgeries) to specific slots.
Sequencing and timing constraints were handled directly at the slot level. 
Each slot allowed for at most one case and enforced its duration plus changeover times.
Idle slots were used to accommodate surgeon availability. 
This approach eliminated the need for pairwise precedence variables.
Another example is \textcite{merkert2020optimal}, which adopted the time-slot concept in their district-heating model.
In their hybrid discrete/continuous-time representation, they first define a continuous-time grid with a fixed number of slots (equal to the number of products) and then make all sequencing and timing decisions at the slot level.
Thus, the MILP derived from the problem \eqref{prob:gdp_ts} provides an efficient modeling framework that yields faster solution times on both operating and energy scheduling problems.

\subsection{Contributions of this work}
\label{sec:contributions}

Mixed‐integer reformulations of GDP models must balance two competing goals when solving them using branch-and-bound solvers: maintaining tight convex relaxations to accelerate convergence, and keeping model size manageable to ensure computational tractability.
Big-M reformulations achieve compactness but often yield weak relaxations. 
In contrast, the HR constructs the convex hull of each disjunction by disaggregating the continuous variables, introducing a copy for every disjunct, and thereby significantly increasing model size.
Note that in the case of multiple disjunctions, the hull reformulation does not yield an \emph{ideal}/\emph{sharp} formulation \cite{vielma2015mixed}. 
This trade-off limits the applicability of existing methods to large-scale scheduling and packing problems, highlighting the need for a new approach that unifies both strengths.
Building on the formulation in \eqref{eq:TSCHC}, we propose a general reaggregated hull reformulation that extends the same key insight: a shared left-hand-side structure within the disjunctions that applies to other GDP instances.

Here, we present the two main contributions of this work toward efficient reformulations of linear GDP models:
\begin{itemize}
  \item \emph{Generalization of Castro and Grossmann’s reformulation technique}: We extend Castro and Grossmann’s reaggregated HR of the time-slot concept in a single-unit scheduling problem approach by introducing a reaggregation strategy that constructs in the original space of continuous variables a MIP reformulation, preserving model compactness and relaxation strength. Similar observations on the strength of convex‐hull reformulations for disjunctive programs date back to \textcite{jeroslow1988alternative} and \textcite{blair1990representation}, whose theoretical results we report here in Theorems \eqref{th:chull} and \eqref{th:reaggchull} and connect to GDP.
  \item \emph{Implementation and evaluation of the GDP reformulation}: We cast the reaggregated hull reformulation in a GDP modeling framework and implement it across a diverse suite of scheduling and strip‐packing test cases, demonstrating its practicality and consistent performance improvements over traditional Big‐M and hull variants. The source code and implementation for this work are publicly available at
\footnote{\label{fn:repo}\url{https://github.com/SECQUOIA/GLDP-reaggregated-hull}.}
\end{itemize}

\section{Methodology: Reaggregation Approach}

Suppose that we work with a GDP model, such as the one given in problem \eqref{prob:GDP}, with the following particular structure: Within a disjunction \(k\), the constraints are linear and share a left‐hand side \(\mathbf{A}_{j'k} = \mathbf{A}_{j''k} = \mathbf{A}_k \quad \forall (j',j'') \in D_k\) within each disjunct, as follows

\begin{equation}
    \label{eq:disj_same_lhs}
    \bigveebar_{j \in D_k} \begin{bmatrix}  Y_{jk} \\ \mathbf{A}_k\mathbf{x}\leq \mathbf{b}_{jk} \end{bmatrix}.
\end{equation}

To continue the discussion, we will list some known results regarding the convex hull of the intersection of disjuncts $F$ characterized by polyhedra $P_{jk}$.

\begin{theorem}[Theorem 3.3 in \textcite{balas1985disjunctive}, also in \textcite{balas1974disjunctive}]
\label{th:chull}
    Let $F = \bigcup_{j \in D_k} P_{jk}, \text{ where } P_{jk} = \{ \mathbf{x} | \mathbf{A}_{jk}\mathbf{x}\leq \mathbf{b}_{jk} \} \text{ and } P_{jk} \neq \emptyset$.
    The convex hull of $F$ is given by the set of \((\hat{\mathbf{x}}_{jk},y_{jk}) \in \mathbb{R}^{n+1}, k \in K, \ j \in D_k\) and \(\mathbf{x} \in \mathbb{R}^{n}\) satisfying the following set of linear inequalities
\begin{align*}
    & \mathbf{x} = \sum_{j \in D_k} \hat{\mathbf{x}}_{jk}\\
    & \mathbf{A}_{jk}\,\hat{\mathbf{x}}_{jk}- \mathbf{b}_{jk}\,y_{jk} \leq 0, \quad\ j \in D_k \\
    & \mathbf{x}^L y_{jk} \leq \hat{\mathbf{x}}_{jk} \leq \mathbf{x}^U y_{jk} \\
    & \sum_{j \in D_k} y_{jk} = 1 \quad \forall k \in K \\
    & y_{jk} \geq 0 \\
\end{align*}

\end{theorem}

\begin{proof}
    See proof of Theorem 3.3 in \textcite{balas1985disjunctive}.
\end{proof}

Before we proceed, we consider the special case where all the linear coefficients in the constraints within the disjunctions are the same.

\begin{corollary}
\label{cor:chull}
        Let $F = \bigcup_{j \in D_k} P_{jk}, \text{ where } P_{jk} = \{ \mathbf{x} | \mathbf{A}_{k}\mathbf{x}\leq \mathbf{b}_{jk} \} \text{ and } P_{jk} \neq \emptyset$.
    The convex hull of $F$ is given by the set of \((\hat{\mathbf{x}}_{jk},y_{jk}) \in \mathbb{R}^{n+1}, k \in K, \ j \in D_k\) and \(\mathbf{x} \in \mathbb{R}^{n}\) satisfying the following set of linear inequalities
\begin{align*}
\label{eq:Hullsimple}
    & \mathbf{x} = \sum_{j \in D_k} \hat{\mathbf{x}}_{jk}\\
    & \mathbf{A}_{k}\,\hat{\mathbf{x}}_{jk}- \mathbf{b}_{jk}\,y_{jk} \leq 0 \quad,\ j \in D_k \\
    & \mathbf{x}^L y_{jk} \leq \hat{\mathbf{x}}_{jk} \leq \mathbf{x}^U y_{jk} \\
    & \sum_{j \in D_k} y_{jk} = 1 \quad \forall k \in K \\
    & y_{jk} \geq 0 \\
\end{align*}
\end{corollary}

Notice that if we enforce a binary condition on the variables $y_{jk}$, and apply it to every disjunction in problem \eqref{prob:GDP}, we recover the known hull reformulation as shown in \eqref{eq:Hull}.

Now we cite a result known in the disjunctive programming literature for decades \cite{jeroslow1988simplification, blair1990representation, vielma2015mixed}, which is still relevant to the discussion as it provides a theoretical background to the strong formulation derived by \textcite{castro2012generalized} and generalized in this work.

\begin{theorem}[Theorem 2.2 in \textcite{jeroslow1988simplification}, Theorem 3 and 4 in \textcite{blair1990representation}]\label{th:reaggchull}
Let $F = \bigcup_{j \in D_k} P_{jk}, \text{ where } P_{jk} = \{ \mathbf{x} | \mathbf{A}_{k}\mathbf{x}\leq \mathbf{b}_{jk} \} \text{ and } P_{jk} \neq \emptyset$.
The convex hull of \(F\) is given by the set of \(y_{jk} \in \mathbb{R}^{1}, k \in K, \ j \in D_k\)  and \(\mathbf{x} \in \mathbb{R}^{n}\) satisfying the following set of linear inequalities

\label{eq:Hullreagg}
\begin{align*}
    & \mathbf{A}_{k}\,\mathbf{x} - \mathbf{b}_{jk}\,y_{jk} \leq 0, \quad\ j \in D_k \\
    & \sum_{j \in D_k} y_{jk} = 1 \quad \forall k \in K \\
    & y_{jk} \geq 0 \\
\end{align*}

\end{theorem} 

\begin{proof}

See proof of Theorem 2.2 in \textcite{jeroslow1988simplification} and Theorem 3 in \textcite{blair1990representation} for sufficient conditions and proof of Theorem 4 in \textcite{blair1990representation} for necessary conditions.

\end{proof}

To provide an intuitive interpretation of this problem, we use a derivation based on the aggregation of constraints.
This is a similar approach to that of \textcite{castro2012generalized} to derive the formulation \eqref{eq:TSCHC} from problem \eqref{prob:gdp_ts}.
We take advantage of the identical linear coefficient \(\mathbf{A}_k\) for every disjunct to use the reaggregation approach.
Summing the individual hull constraints \(\mathbf{A}_k \hat{\mathbf{x}}_{jk} \leq b_{jk} y_{jk}\) and using \(\mathbf{x} = \sum_{j \in D_k} \hat{\mathbf{x}}_{jk}\) yield the single reaggregated inequality,

\begin{equation*}
    \mathbf{A}_k \mathbf{x} \leq \sum_{j \in D_k} \mathbf{b}_{jk} y_{jk} \quad \forall k \in K.
\end{equation*}

By replacing all per‐disjunct copies \(\hat{\mathbf{x}}_{jk}\) with original continuous variables \(\mathbf{x}\) and maintaining only binary variables \(y_{jk}\) and related constraints, the reaggregated hull formulation eliminates extra continuous variables and their associated linking constraints.
This reduction in model size comes without sacrificing tightness.
Since all disjuncts share \(\mathbf{A}_k\), summing the individual constraints
\(\mathbf{A}_k\hat{\mathbf{x}}_{jk}\le \mathbf{b}_{jk}y_{jk}\) under
\(\mathbf{x}=\sum_j\hat{\mathbf{x}}_{jk}\) yields the single inequality
\(\mathbf{A}_k\mathbf{x}\le\sum_j \mathbf{b}_{jk}\,y_{jk}\).

The inequality of the reaggregated method describes the exact convex hull of the disjunctive set, yielding a compact MIP with strong LP relaxation and fewer variables and constraints than the complete HR \parencite{balas1974disjunctive,balas1985disjunctive, jeroslow1988simplification, blair1990representation}.

The reaggregated hull reformulation (RHR) for a problem \eqref{prob:GDP} with shared linear coefficients as in Eq. \eqref{eq:disj_same_lhs} then yields

\begin{equation}
\label{eq:HR-reagg}
\tag{RHR}
R= \left\{ (\mathbf{x},\mathbf{y}) :
\begin{aligned}
    & \mathbf{A}_k \mathbf{x} \leq \sum_{j \in D_k} \mathbf{b}_{jk} y_{jk}  \quad \forall k \in K, \ j \in D_k \\
    & \sum_{j \in D_k} y_{jk} = 1  \quad \forall k \in K.
\end{aligned}
\right\}
\end{equation}

\section{Results and Discussion}
\label{sec:results}

This section demonstrates the efficiency of the reaggregated hull formulation through several case studies.
The GDP models were implemented in Pyomo 6.9.0 \cite{bynum2021pyomo}.
Then, the GDP models are reformulated into MILP models using aggregated hull reformulation as well as Big-M and HR.
Each case study was executed on a Linux cluster equipped with 48 AMD EPYC 7643 2.3 GHz CPUs and 1 TB of RAM, with each run restricted to a single thread.
All MILP formulations were solved using Gurobi 11.0 \cite{gurobi}, and a relative optimality gap of 0.01\% was enforced for all runs.
We also ran the same experiments with SCIP~9.2.3 and HiGHS~1.10 under identical settings; detailed per-instance results and performance profiles are reported in the Supporting Information in \S~C and D, respectively \parencite{Achterberg2009, huangfu2018parallelizing}.
The implementation of the reformulations and instances presented in this work is available in the public repository \cite{secquoia_gdp_sched}. 

\subsection{Single-Unit Scheduling}

We reproduced the single-unit batch plant scheduling cases from \textcite{castro2012generalized} using all three scheduling concepts, General Precedence \eqref{prob:gdp_gp}, Immediate Precedence \eqref{eq:gdp_ip}, and Time-Slot \eqref{prob:gdp_ts}, and reformulated each GDP model as an MILP via both Big-M and hull reformulations.

As seen in \eqref{eq:TSCHC}, \textcite{castro2012generalized} derived the reaggregated hull reformulation of the problem \eqref{prob:gdp_ts}.
Below, we demonstrate that a similar reformulation can be achieved by modifying the problem \eqref{prob:gdp_gp} as follows.
The algebraic constraints \eqref{eq:gp_lb} and \eqref{eq:gp_ub} yield the bounds

\begin{equation}
    r_i \leq x_i \leq d_i - p_i \quad \forall i \in I,
\end{equation}
which restricts the feasible start times of each job \(i\).

\noindent From these inequalities, the bounds of \(x_i - x_j\) are

\begin{equation}
    r_i - (d_j - p_j) \leq x_i -x_j \leq (d_i - p_i) - r_j \quad \forall i \in I, \ \forall j \in I, \ i<j.
\end{equation}
The bounds created by the algebraic constraints can be applied to the disjunctions.
The disjunction will be reformulated as

\begin{align*}
&\begin{bmatrix}
    Y_{ij}\\[2pt]
    r_i - (d_j - p_j) \le x_i - x_j
      \le \min\bigl(-p_i,\;(d_i - p_i) - r_j\bigr)
\end{bmatrix}
\veebar
\\[6pt]
&\begin{bmatrix}
    \neg Y_{ij}\\[2pt]
    \max\bigl(p_j,\;r_i - (d_j - p_j)\bigr)
      \le x_i - x_j
      \le (d_i - p_i) - r_j
\end{bmatrix}.
\end{align*}
As each GP disjunct shares identical linear coefficients, the hull constraints can be summed and reaggregated into a single inequality, yielding the RHR for the \eqref{prob:gdp_gp} model.
As a side note, the procedure of including the global bound constraints in the disjunctions can also be interpreted as a basic step in disjunctive programming~\cite{balas2018disjunctive}.
This is a technique that increases the relaxation strength of a disjunctive formulation by introducing more constraints when potentially merging disjunctions or including constraints in the existing ones.
In this approach, this basic step is also justified as it led to a structure where the reaggregation technique is possible.

In contrast, the \eqref{eq:gdp_ip} model’s disjuncts use different linear coefficients in its constraints.  
We therefore apply the RHR approach only to \eqref{prob:gdp_gp} and \eqref{prob:gdp_ts}.
The ten examples (Ex. 1–Ex. 10) used for benchmarking are single‐unit scheduling instances originally introduced in \textcite{castro2012generalized}, whose detailed job-processing times and precedence structures are listed in the Supporting Information A and implemented in the repository \cite{secquoia_gdp_sched}.
For each of these instances, we tested the eight different MILP reformulations arising from the combination of the various GDP models with the Big-M, hull, or regagggregated hull reformulations.
We denote each combination of model and reformulation as a variant.
The ten single-unit instances have job counts \(|I| = 15, 20, 25, 30\), but the instance size also varies with processing times, release times, and due times, so cases with the same \(|I|\) can differ substantially.
We used a 900-second time limit per instance and compared their computational performance.


\begin{table}[t]
\centering
\caption{Benchmark results for the scheduling problem across 10 instances, each with a different number of jobs \( |I| \in \{15, 20, 25, 30\} \). 
        Each instance is solved using multiple formulations: General Precedence \eqref{prob:gdp_gp}, Immediate Precedence \eqref{eq:gdp_ip}, and Time Slot \eqref{prob:gdp_ts}, reformulated as a \ref{prob:mip} with Big-M \eqref{eq:BigM}, Hull \eqref{eq:Hull}, and Reaggregated Hull \eqref{eq:HR-reagg} reformulations.
        The benchmark compares the computational performance by reporting the time to prove optimality within 0.01\% in seconds with a time limit of 900 seconds. Entries of the form 900+ \((\gamma)\) indicate the run hit the limit with a remaining optimality gap of \(\gamma \%\); 900+ \((\infty)\) means no feasible incumbent was found.}
\label{tab:benchmark_results}
\setlength{\tabcolsep}{6pt}
\begin{tabular}{lrrrrr}
\toprule
Problem & \(|I|\) & Optimum & GP\_BM & GP\_HR & GP\_RHR \\
\midrule
Ex. 1  & 15 & 430 & 25.89 & 144.63 & 30.61 \\
Ex. 2  & 15 & 438 & 6.54  & 39.97  & 3.83  \\
Ex. 3  & 20 & 559 & 0.24  & 0.58   & 0.21  \\
Ex. 4  & 20 & 664 & 385.59& 900+ (5.87) & 459.09 \\
Ex. 5  & 20 & 680 & 54.96 & 219.75 & 56.59 \\
Ex. 6  & 25 & 739 & 25.30 & 175.28 & 26.31 \\
Ex. 7  & 25 & 700 & 866.15 & 900+ (5.86) & 900+ (0.71) \\
Ex. 8  & 25 & 772 & 900+ (3.89) & 900+ (5.44) & 900+ (3.50) \\
Ex. 9  & 30 & 853 & 900+ (25.56) & 900+ (30.48) & 900+ (33.28) \\
Ex. 10 & 30 & 929 & 900+ (19.16) & 900+ (25.19) & 900+ (20.45) \\
\bottomrule
\end{tabular}

\vspace{0.8em}

\begin{tabular}{lrrrrrrrr}
\toprule
Problem & \(|I|\) & Optimum & IP\_BM & IP\_HR & TS\_BM & TS\_HR & TS\_RHR \\
\midrule
Ex. 1  & 15 & 430 & 900+ (6.05) & 0.63   & 19.76               & 0.37 & 0.12 \\
Ex. 2  & 15 & 438 & 900+ (47.26) & 0.61   & 813.6               & 0.81 & 0.11 \\
Ex. 3  & 20 & 559 & 19.63      & 49.62  & 94.3                & 6.75 & 0.72 \\
Ex. 4  & 20 & 664 & 900+ (37.5)  & 25.54  & 900+ (54.97)          & 1.58 & 0.37 \\
Ex. 5  & 20 & 680 & 900+ (25.29) & 28.60  & 900+ (54.56)          & 1.82 & 0.22 \\
Ex. 6  & 25 & 739 & 900+ (21.24) & 28.67  & 900+ (59.64)\(^{*}\)  & 2.40 & 0.58 \\
Ex. 7  & 25 & 700 & 900+ \((\infty)\) & 900+ \((\infty)\) & 900+ \((\infty)\) & 2.10 & 1.54 \\
Ex. 8  & 25 & 772 & 900+ \((\infty)\) & 24.47 & 900+ \((\infty)\)   & 3.36 & 0.50 \\
Ex. 9  & 30 & 853 & 900+ \((\infty)\) & 24.12 & 900+ \((\infty)\)   & 2.57 & 0.43 \\
Ex. 10 & 30 & 929 & 900+ \((\infty)\) & 449.05 & 900+ \((\infty)\)  & 7.81 & 0.53 \\
\bottomrule
\end{tabular}

\vspace{0.3em}
\noindent\footnotesize\emph{\emph{Note}: Ex.~6 (TS\_BM) timed out at 900~s; best incumbent objective \(=769\) (suboptimal).}
\end{table}

Table~\ref{tab:benchmark_results} presents benchmark results for a series of scheduling problem instances formulated as various Generalized Disjunctive Programming (GDP) models and solved using different reformulation strategies.
Each table entry reports the solve time (in seconds) for a given method–instance pair.
If a run reached the 900-second limit without proving optimality, we record it as 900+ \( (\gamma)\), where \(\gamma\) is the solver-reported optimality gap (in \%) at the termination.
The notation 900+ \((\infty)\) indicates that no feasible incumbent was found, resulting in the gap being undefined.

In general, the table shows that the time-slot formulations (both \(\text{TS}\_\text{HR}\) and \(\text{TS}\_\text{RHR}\)) dominate the other variants, as the solver Gurobi solved every instance in less time. 
Within each model, the reaggregated hull variants (\(\text{GP}\_\text{RHR}\) and \(\text{TS}\_\text{RHR}\)) consistently beat their HR counterparts (\(\text{GP}\_\text{HR}\) and \(\text{TS}\_\text{HR}\)).



The time‐slot formulations (\(\text{TS}\_\text{HR}\) and \(\text{TS}\_\text{RHR}\))  were solved to proven optimality within 900 seconds for all ten instances. 
On the other hand, the Big-M time-slot variant \((\text{TS}\_\text{BM})\) frequently terminated at the time limit with a remaining gap; on Ex. 6 it timed out with a suboptimal incumbent (best = 769) and reported no feasible incumbent on the largest instances (Ex. 7–10).
Within the immediate precedence \eqref{eq:gdp_ip} variants, \((\text{IP}\_\text{HR})\) was solved to proven optimality on 9/10 instances within 900 seconds, with only Ex.~7 timing out with no feasible incumbent; in contrast, \((\text{IP}\_\text{BM})\) proved optimality only on Ex.~3 and otherwise reached the time limit, including no feasible incumbent on Ex. 7–10.
For the general precedence \eqref{prob:gdp_gp} formulations, feasible incumbents were obtained on every instance, but several runs, such as Ex. 8-10, reached the time limit without closing the optimality gap. 
Thus, \(\text{TS}\_\text{HR}\) and \(\text{TS}\_\text{RHR}\) closed the optimality gap on all instances within the time limit, the GP formulations frequently timed out with nonzero gaps, and only the \(\text{IP}\_\text{HR}\) formulation was solved most instances with in the limits.

Within the time-slot formulations, the reaggregated hull \((\text{TS}\_\text{RHR})\) consistently reduced runtimes relative to hull \((\text{TS}\_\text{HR})\) and was solved on all ten instances within 900 seconds.
By contrast, the Big-M reformulations of the time-slot concept (\(\text{TS}\_\text{BM}\)) frequently reached the time limit with a remaining gap, including cases with no feasible incumbent.
Even with tight $M$ coefficients, the weak relaxation led to slow convergence.

For the general-precedence \eqref{prob:gdp_gp} formulations, the reaggregated hull \((\text{GP}\_\text{RHR})\) improved over the hull reformulations \((\text{GP}\_\text{HR})\) by delivering shorter runtimes and smaller final gaps when being solved with Gurobi.
On the hardest instances, \(\text{GP}\_\text{RHR}\) was broadly comparable to the Big-M variant (\(\text{GP}\_\text{BM}\)).

Similar observations were made when solving these problems with open-source MIP solvers SCIP and Highs, and the details are presented in the Supporting Information.

Revisiting the single unit scheduling problem by \textcite{castro2012generalized} confirms that the time slot concept of the model using the reaggregated HR (\(\text{TS}\_\text{RHR}\)) leads to the best performing variant when solved using Gurobi.
Extending the reaggregated HR to the General Precedence \eqref{prob:gdp_gp} yields \(\text{GP}\_\text{RHR}\).
Despite the worse performance compared to ({\(\text{TS}\_\text{RHR}\)), the general precedence model, \(\text{GP}\_\text{RHR}\) outperforms its traditional HR variant (\(\text{GP}\_\text{HR}\)) and matched it Big-M alternative (\(\text{GP}\_\text{BM}\)).
Even with advances in solver algorithms and hardware, the strong MIP reformulations introduced by \textcite{castro2012generalized} continue to enable modern solvers to achieve superior performance.

\subsection{Strip-Packing Model}

The two‐dimensional strip‐packing problem was first posed by \textcite{baker1980orthogonal} and consists of placing a given set of non-rotatable axis-aligned rectangles, each of width \(L_i\) and height \(H_i\), into a fixed-width strip \(W\) to minimize its total length \(l_t\).  
Feasible solutions assign the continuous position variables for each rectangle \((x_i,y_i)\) subject to boundary constraints and pairwise non‐overlap requirements.
An example of an optimal layout for the strip‐packing problem is shown in Figure~\ref{fig:strip-packing}.
\textcite{castro2012time} were the first to bring planning-inspired time representation ideas to this problem, introducing discrete and continuous-time MILP models that reduce model size while preserving solution quality. 


\begin{figure}[ht]
    \centering

    \begin{subfigure}[b]{0.50\linewidth}
        \centering
        \includegraphics[width=\linewidth]{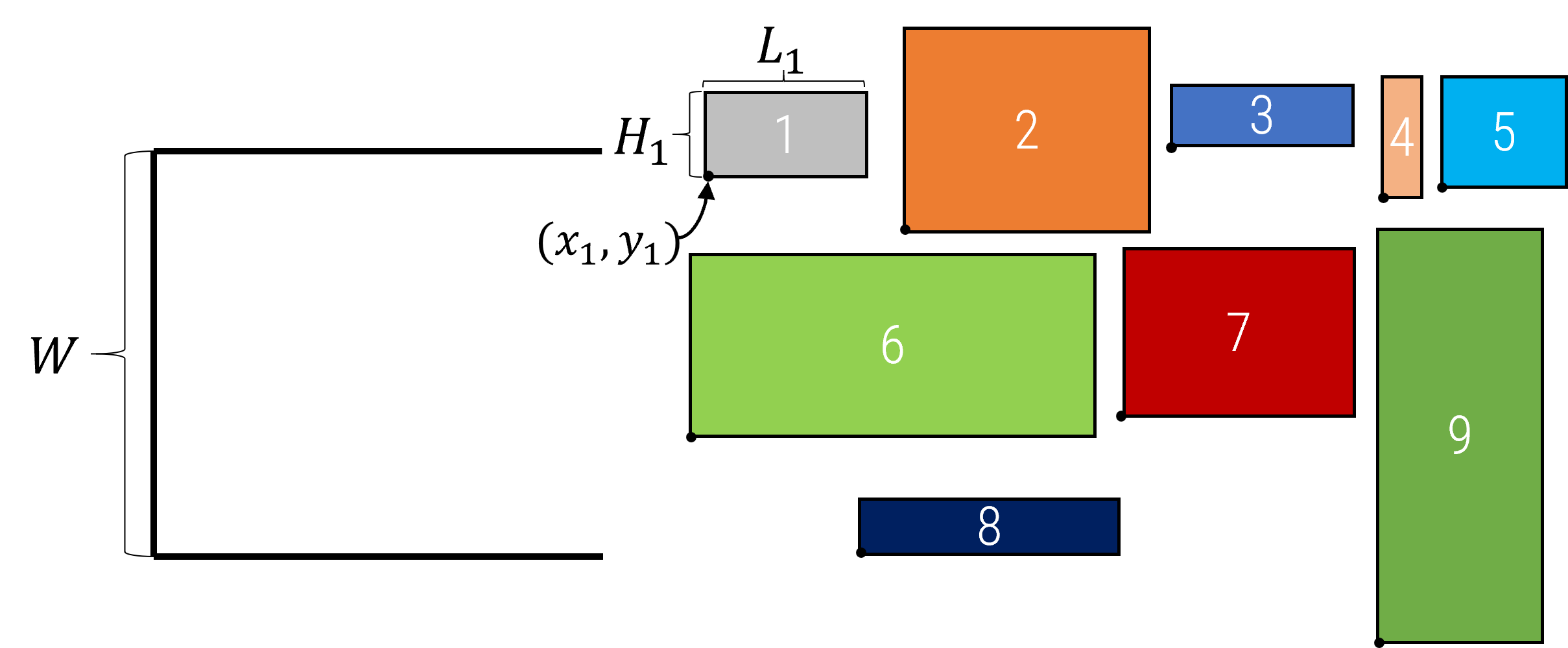}
        \caption{}
        \label{fig:modified_strip_a}
    \end{subfigure}
    \hfill
    \begin{subfigure}[b]{0.35\linewidth}
        \centering
        \includegraphics[width=\linewidth]{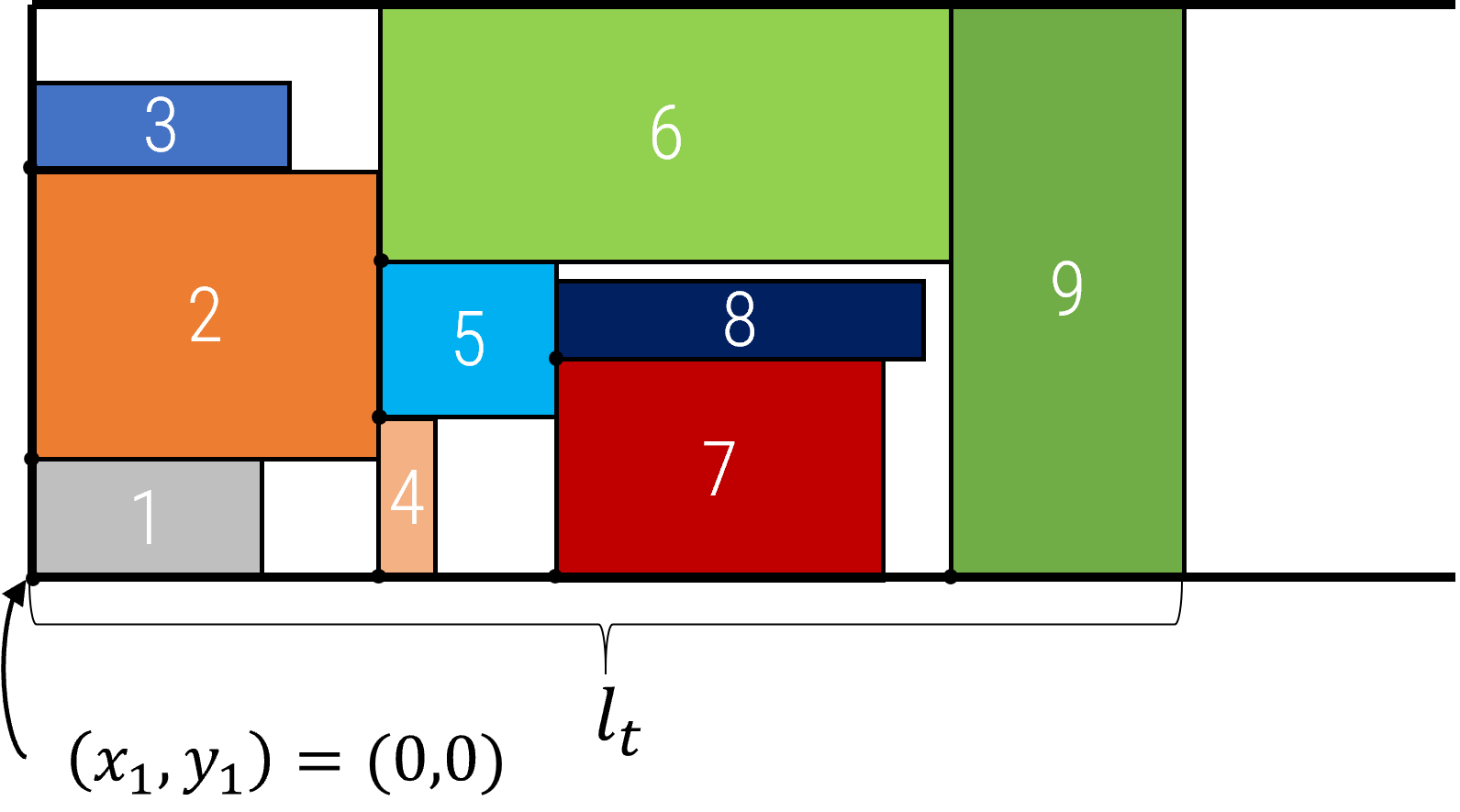}
        \caption{}
        \label{fig:modified_strip_b}
    \end{subfigure}

    \caption{Two-dimensional strip-packing problem. (a) Schematic of the strip-packing setup, where unrotatable rectangles are placed in a strip of fixed width \(W\) and unbounded length; variables \((x_i, y_i)\) specify the bottom-left corner of each rectangle. (b) Example of a feasible packing configuration, where the resulting strip length \(\ell_t\) is minimized subject to boundary and non-overlap constraints.}
    \label{fig:strip-packing}
\end{figure}

\textcite{sawaya2005cutting} were the first to cast the two‐dimensional strip‐packing problem in a generalized disjunctive programming (GDP) framework, which we employ here.
We used their GDP formulation here to test the reaggregated hull \cite{sawaya2005cutting}.
This GDP model provides a unified framework for the two-dimensional strip-packing problem by embedding both the geometric placement and the combinatorial non-overlap requirements directly into a logical disjunction structure.
The objective is to minimize the total strip length \(l_t\).

\begin{align*}
\tag{S-original}
\label{strip:original}
    \min_{\substack{l_t, \mathbf{x}, \mathbf{y}, \\ \mathbf{Z}^1, \mathbf{Z}^2}} \quad & l_t \\
    \text{s.t.} \quad & l_t \geq x_i + L_i, \quad \forall \ i \in N, \\
    &
  \begin{bmatrix}
    Z^1_{ij}\\
    x_i + L_i \le x_j
  \end{bmatrix}
  \;\vee\;
  \begin{bmatrix}
    Z^1_{ji}\\
    x_j + L_j \le x_i
  \end{bmatrix} 
  \;\vee\;
  \begin{bmatrix}
    Z^2_{ij}\\
    y_i - H_i \ge y_j\\
  \end{bmatrix}
  \;\vee\;
  \begin{bmatrix}
    Z^2_{ji}\\
    y_j - H_j \ge y_i\\
  \end{bmatrix}
    &\forall\,i,j\in N,\;i<j \\
    &Z^1_{ij} \,\veebar\, Z^1_{ji} \,\veebar\, Z^2_{ij} \,\veebar\, Z^2_{ji}, 
      \quad \forall\,i,j \in N,\;i<j,\\
    &x_i \in [0, UB - L_i] \subseteq \mathbb{R}, 
      \quad \forall\,i \in N,\\
    &y_i \in [H_i, W] \subseteq \mathbb{R}, 
      \quad \forall\,i \in N,\\
    &Z^1_{ij},\,Z^2_{ij} \in \{\text{False},\text{True}\}, 
      \quad \forall\,i,j \in N,\;i \neq j.
\end{align*}

Given this standard strip-packing GDP formulation, the continuous variables \(x_i\) and \(y_i\) refer to the horizontal and the vertical coordinates of the bottom-left corner of rectangle \(i\) within the strip.
In addition, each rectangle \(i\) is bounded by \(0 \leq x_i \leq UB - L_i\), and \(H_i \leq y_i \leq W\) to ensure its locations within a fixed width strip \(W\).
For every unordered pair \((i,j)\), the four Boolean variables \(Z_{ij}^1, Z_{ji}^1, Z_{ij}^2, Z_{ji}^2\) index the mutually exclusive spatial relations; ``\(i\) to the left of \(j\)'', ``\(j\) to the left of \(i\)'', ``\(i\) to the above of \(j\)'', ``\(j\) to the above of \(i\)''.
The exclusive-or constraint \(\veebar\) enforces that exactly one of these relations is activated.
Therefore, the disjunctions guarantee that no two rectangles overlap within the strips.

\textcite{trespalacios2017symmetry} presented a symmetry-breaking GDP formulation for the two-dimensional strip-packing problem as follows

\begin{align*}
\tag{S-symbreak}
\label{strip:tres}
\min_{\substack{l_t, \mathbf{x}, \mathbf{y}, \\ \mathbf{Z}^1, \mathbf{Z}^2}} \quad & l_t \\
\text{s.t.}\quad 
& l_t \;\ge\; x_i + L_i,
\quad \forall \ i\in N,\\
& 
  \begin{bmatrix}
    Z^1_{ij}\\
    x_i + L_i \le x_j
  \end{bmatrix}
  \;\vee\;
  \begin{bmatrix}
    Z^1_{ji}\\
    x_j + L_j \le x_i
  \end{bmatrix}
  \;\vee\;
  \begin{bmatrix}
    Z^2_{ij}\\
    y_i - H_i \ge y_j\\
    x_i + L_i \ge x_j\\
    x_j + L_j \ge x_i
  \end{bmatrix}
  \;\vee\;
  \begin{bmatrix}
    Z^2_{ji}\\
    y_j - H_j \ge y_i\\
    x_i + L_i \ge x_j\\
    x_j + L_j \ge x_i
  \end{bmatrix}
&\forall\,i,j\in N,\;i<j \\
&Z^1_{ij} \,\veebar\, Z^1_{ji} \,\veebar\, Z^2_{ij} \,\veebar\, Z^2_{ji}, 
  \quad \forall\,i,j \in N,\;i<j,\\
&x_i \in [0, UB - L_i] \subseteq \mathbb{R}, 
  \quad \forall\,i \in N,\\
&y_i \in [H_i, W] \subseteq \mathbb{R}, 
  \quad \forall\,i \in N,\\
&Z^1_{ij},\,Z^2_{ij} \in \{\text{False},\text{True}\}, 
  \quad \forall\,i,j \in N,\;i \neq j.
\end{align*}

In \eqref{strip:tres}, each vertical disjunct includes both the usual \(y\)-separation and additional $x$-separation constraints to avoid symmetric packings.
By enforcing horizontal bounds together with the vertical disjunctions, the \eqref{strip:tres} model eliminates redundant, equivalent solutions and enhances the performance of MILP solvers.
Specifically, in the third and fourth disjunctive matrices, the extra constraints
\begin{equation*}
x_i + L_i \geq x_j\quad \text{and} \quad x_j + L_j \geq x_i
\end{equation*}
impose a left-to-right ordering when rectangles are stacked vertically.
Therefore, these constraints break mirror-image symmetries and reduce redundant search space.
This formulation proved to be more amenable to MILP solvers, leading to improved numerical performance \parencite{trespalacios2017symmetry}.

Both strip-packing formulations, standard \eqref{strip:original} and symmetry-breaking \eqref{strip:tres} of \textcite{trespalacios2017symmetry}, admit a strengthening when cast under the reaggregated hull framework.
In particular, observe that every rectangle \(i\) is intrinsically bounded by
\begin{equation*}
    0 \leq x_i \leq UB - L_i, \quad H_i \leq y_i \leq W,
\end{equation*}
where \(UB\) is an upper bound on the strip length and \(W\) its fixed width.
Given pair \((i,j)\) that \(i<j\), the general constraints can form bounds
\begin{equation*}
    -UB + L_j \leq x_i - x_j \leq UB - L_i, \quad -W + H_i \leq y_i - y_j \leq W - H_j.
\end{equation*}
By enforcing these global difference bounds within each disjunct, the coefficients on \((x_i - x_j)\) and \((y_i-y_j)\) become identical across all spatial relations.
Once again, this would be equivalent to a basic step between the global constraints and all the disjuncts \cite{balas2018disjunctive}.

The two formulations that we are comparing here result in

\begin{align*}
\tag{S0}
\label{strip:s0}
    \min_{\substack{l_t, \mathbf{x}, \mathbf{y}, \\ \mathbf{Z}^1, \mathbf{Z}^2}} \quad & l_t \\
    \text{s.t.} \quad & l_t \geq x_i + L_i, \quad \forall \ i \in N, \\
    &
  \begin{bmatrix}
    Z^1_{ij}\\
    x_i - x_j \leq -L_i\\
    y_i - y_j \geq -W + H_i\\
    y_i - y_j \leq W - H_j \\
    x_i - x_j \geq -UB + L_j
  \end{bmatrix}
  \;\vee\;
  \begin{bmatrix}
    Z^1_{ji}\\
    x_i - x_j \geq L_j\\
    y_i - y_j \geq -W + H_i\\
    y_i - y_j \leq W - H_j \\
    x_i - x_j \leq UB - L_i
  \end{bmatrix} \\
  &\;\vee\;
  \begin{bmatrix}
    Z^2_{ij}\\
    y_i - y_j \ge H_i\\
    x_i - x_j \geq -UB + L_j \\
    x_i - x_j \leq UB - L_i\\
    y_i - y_j \leq W - H_j
  \end{bmatrix}
  \;\vee\;
  \begin{bmatrix}
    Z^2_{ji}\\
    y_i - y_j \leq - H_j\\
    x_i - x_j \geq -UB + L_j \\
    x_i - x_j \leq UB - L_i \\
    y_i - y_j \geq -W + H_i
  \end{bmatrix}
    &\forall\,i,j\in N,\;i<j \\
    &Z^1_{ij} \,\veebar\, Z^1_{ji} \,\veebar\, Z^2_{ij} \,\veebar\, Z^2_{ji}, 
      \quad \forall\,i,j \in N,\;i<j,\\
    &x_i \in [0, UB - L_i] \subseteq \mathbb{R}, 
      \quad \forall\,i \in N,\\
    &y_i \in [H_i, W] \subseteq \mathbb{R}, 
      \quad \forall\,i \in N,\\
    &Z^1_{ij},\,Z^2_{ij} \in \{\text{False},\text{True}\}, 
      \quad \forall\,i,j \in N,\;i \neq j,
\end{align*}

which is a version of \eqref{strip:original} with additional constraints and 
\begin{align*}
\tag{S1}
\label{strip:s1}
    \min_{\substack{l_t, \mathbf{x}, \mathbf{y}, \\ \mathbf{Z}^1, \mathbf{Z}^2}} \quad & l_t \\
    \text{s.t.} \quad & l_t \geq x_i + L_i, \quad \forall \ i \in N, \\
    &
  \begin{bmatrix}
    Z^1_{ij}\\
    x_i - x_j \leq -L_i\\
    y_i - y_j \geq -W + H_i\\
    y_i - y_j \leq W - H_j \\
    x_i - x_j \geq -UB + L_j
  \end{bmatrix}
  \;\vee\;
  \begin{bmatrix}
    Z^1_{ji}\\
    x_i - x_j \geq L_j\\
    y_i - y_j \geq -W + H_i\\
    y_i - y_j \leq W - H_j \\
    x_i - x_j \leq UB - L_i
  \end{bmatrix} \\
  &\;\vee\;
  \begin{bmatrix}
    Z^2_{ij}\\
    y_i - y_j \ge H_i\\
    x_i - x_j \geq \max \{-UB + L_j, -L_i\} \\
    x_i - x_j \leq \min \{UB - L_i, L_j\}\\
    y_i - y_j \leq W - H_j
  \end{bmatrix}
  \;\vee\;
  \begin{bmatrix}
    Z^2_{ji}\\
    y_i - y_j \leq - H_j\\
    x_i - x_j \geq \max \{-UB + L_j, -L_i\} \\
    x_i - x_j \leq \min \{UB - L_i, L_j\}\\
    y_i - y_j \geq -W + H_i
  \end{bmatrix}
    &\forall\,i,j\in N,\;i<j \\
    &Z^1_{ij} \,\veebar\, Z^1_{ji} \,\veebar\, Z^2_{ij} \,\veebar\, Z^2_{ji}, 
      \quad \forall\,i,j \in N,\;i<j,\\
    &x_i \in [0, UB - L_i] \subseteq \mathbb{R}, 
      \quad \forall\,i \in N,\\
    &y_i \in [H_i, W] \subseteq \mathbb{R}, 
      \quad \forall\,i \in N,\\
    &Z^1_{ij},\,Z^2_{ij} \in \{\text{False},\text{True}\}, 
      \quad \forall\,i,j \in N,\;i \neq j,
\end{align*}

which is \eqref{strip:tres} with additional constraints strengthened by the symmetry-breaking bounds.
Both formulations include additional linear inequalities in the disjunctions to satisfy the reaggregation condition in the RHR.
Thus, RHR can be applied directly to both \eqref{strip:s0} and \eqref{strip:s1} strip-packing models.

We generate ten instances for each count of rectangles, \(n=5,\dots,19\) (in total, 150 instances), and apply three reformulations, Big-M, HR, and RHR, to both \eqref{strip:s0} and \eqref{strip:s1}. 
Each variant was solved with a per-instance time limit of 1,800 seconds.

\begin{figure}[htbp]
  \centering
  \includegraphics[width=\textwidth]{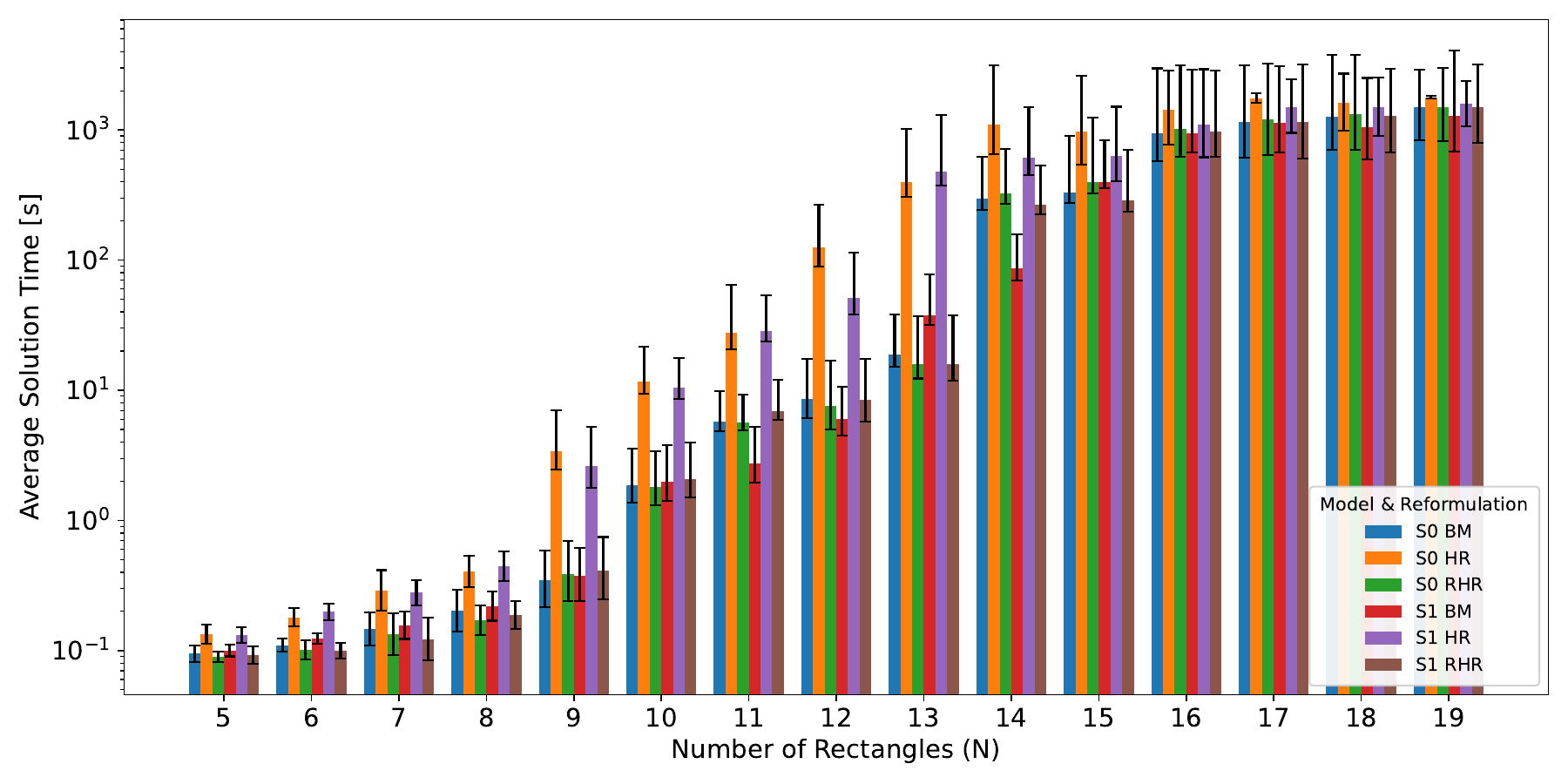}
  \caption{Benchmark run times for strip packing formulations (Gurobi). The \(x\)-axis label “Number of Rectangles” indicates the problem size. At each size, six bars are shown in the following order: \ref{strip:s0} BM, \ref{strip:s0} HR, \ref{strip:s0} RHR, \ref{strip:s1} BM, \ref{strip:s1} HR, and \ref{strip:s1} RHR, where \ref{strip:s0} is the standard strip-packing model and \ref{strip:s1} is the symmetry-breaking model. Bars report the average solution times (seconds) over ten instances per size, with error bars indicating \(\pm\)1 standard deviation.}
  \label{fig:strip-benchmark-run-times}
\end{figure}

The solution times for the ten different cases were aggregated into sample means and standard deviations for each reformulation and the number of strips, yielding Figure \ref{fig:strip-benchmark-run-times}.
The bar plots show the mean solve time for each reformulation across ten random instances at each problem size (number of rectangles on the \(x\)-axis).  
The black error bars denote one standard deviation.
\ref{strip:s0}  denotes the standard GDP formulation with only basic horizontal and vertical separation. 
In contrast, \ref{strip:s1} denotes the symmetry-breaking variant with added horizontal ordering constraints to eliminate mirror-image packings.
Big-M, HR, and RHR were applied to both models, yielding six variants.

In the original strip-packing formulation, the reaggregated hull variant consistently achieves lower mean solve times than the standard hull reformulation across all instance sizes, because the reaggregated formulation uses fewer variables and constraints per packing choice.
Similarly, in \eqref{strip:s1} symmetry-breaking model, the reaggregated hull approach outperforms its classic hull counterpart at every problem size, demonstrating that these cuts continue to speed up solving even when additional horizontal ordering rules are introduced.
Notably, the base formulation \eqref{strip:s0} and the symmetry-breaking formulation \eqref{strip:s1} exhibit virtually identical solve times, indicating that the additional horizontal ordering constraints in \eqref{strip:s1} have a negligible impact on performance, in contrast to speedups reported by \textcite{trespalacios2017symmetry} from symmetry breaking.
Moreover, across both the \eqref{strip:s0} and \eqref{strip:s1} models, the reaggregated hull variant consistently runs slightly faster than its Big-M counterpart.

In Figure \ref{fig:strip-benchmark-run-times}, the error bars represent \(\pm 1\) standard deviation of the solve times over ten random instances for each problem size.
For the largest problems, i.e., 17 rectangles, the HR often reaches the 1,800-second time cap on every instance, resulting in error bars that span up to the time limit.
In contrast, the reaggregated hull variants (both \eqref{strip:original} and \eqref{strip:tres}) converge before the cutoff in every instance, and their much shorter error bars reflect consistency in the solve times.

\begin{figure}
    \centering
    \includegraphics[width=0.8\linewidth]{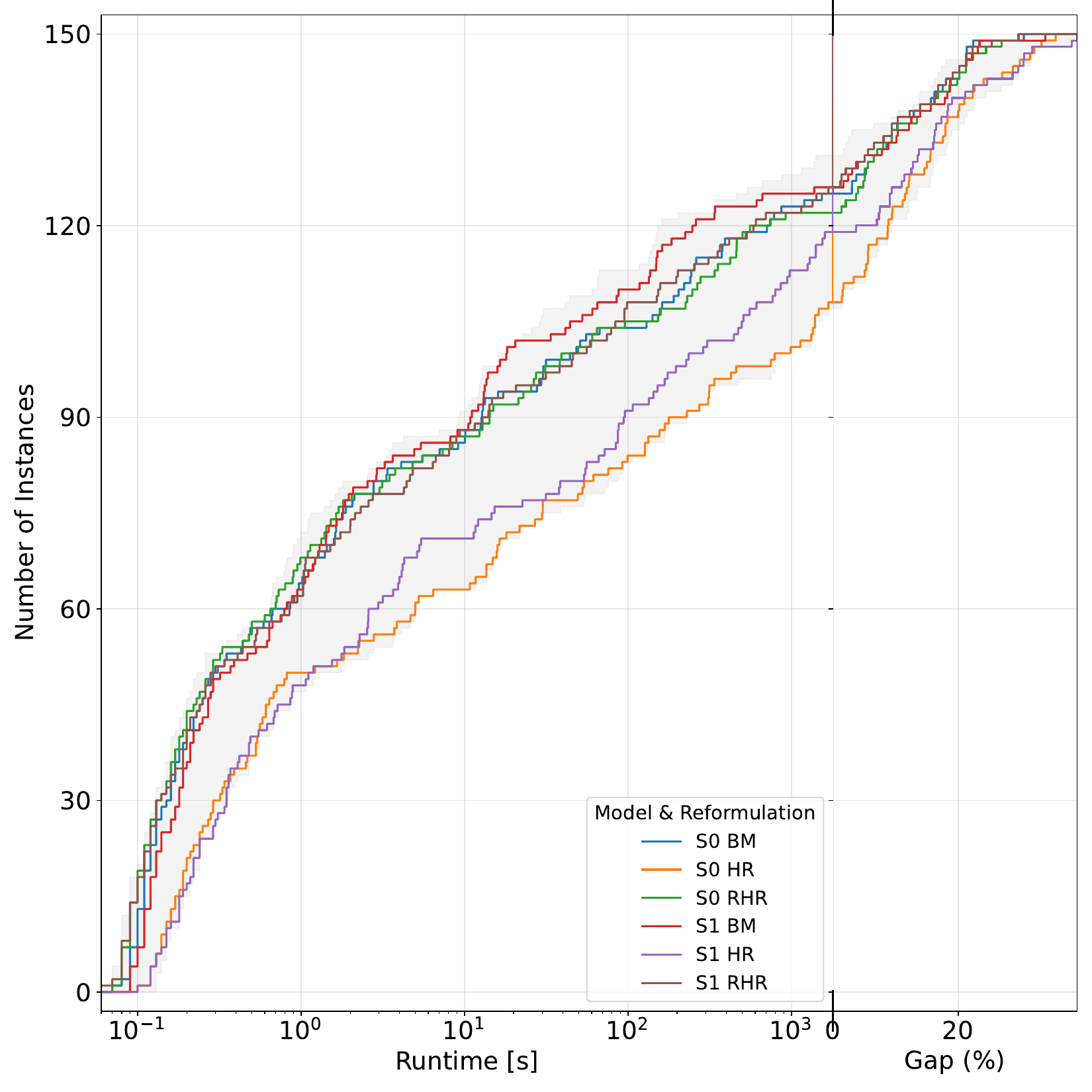}
    \caption{Absolute performance profile of the strip-packing model (Gurobi). Curves compare six variants: \ref{strip:s0} BM, \ref{strip:s0} HR, \ref{strip:s0} RHR, \ref{strip:s1} BM, \ref{strip:s1} HR, and \ref{strip:s1} RHR. The left panel reports the cumulative number of instances solved within runtime \(t\) seconds (cutoff 1,800 seconds). The right panel reports the cumulative number of instances whose optimality gap is at most \(\gamma\) \((\%)\) within the same limit. The shaded gray region spans the virtual best and virtual worst envelopes, formed by taking for each instance the minimum and maximum performance across all variants.}
    \label{fig:strip-performance-profile}
\end{figure}

Figure \ref{fig:strip-performance-profile} shows the absolute performance profile over all 150 randomly generated strip-packing models, comparing the six different reformulations.
On the left, each curve shows how many of the ten instances a method solves within time \(t\) up to the 1,800-second cutoff \cite{Dolan2002}.
On the right, each curve plots, for each optimality‐gap threshold \(\gamma\) (\%), the cumulative number of instances whose gap is reduced to at most \(\gamma\) within the time limit.
Any method that fails to solve an instance or to find a solution to compute a remaining optimality gap within the time limit is recorded as “aborted” and does not contribute beyond that point, causing its curve to plateau.
The shaded grey region spans the virtual best and virtual worst frontiers, defined by the per instance minimum and maximum performance across all reformulations.
Overall, it is possible to compare methods based on both solution time and optimality gap from the absolute performance profile.

Among the base formulations \eqref{strip:s0} strip packing models, the reaggregated hull variant outperforms both the Big-M and hull reformulations.
The reaggregated hull variants result in the most instances being solved within the time limit and yield the smallest optimality gap when optimality cannot be proven within that time limit.
The \eqref{strip:s0} hull reformulation, despite its theoretically stronger relaxation compared to Big-M, leads to the weakest performance when solving the problems with Gurobi among the three original reformulations.
Similar observations are made when solving these problems with the solvers SCIP and Highs, as presented in the Supporting Information, albeit with slower performance across all instances compared to Gurobi.


Overall, the strip-packing case studies demonstrate that the RHR outperforms the HR and makes it competitive with the Big-M reformulation.


\section{Conclusion and Future Work}

The computational results demonstrate that the reaggregated hull reformulation delivers an improved performance across both case studies compared to the extended hull reformulation.
Moreover, it makes this theoretically stronger formulation practically competitive with the more traditionally used Big-M formulation.
In the single-unit scheduling benchmark, the problem \eqref{prob:gdp_ts} reformulated as \ref{eq:HR-reagg} solved ten instances well under the 900-second limit, remarkably the fastest compared to other variants.
These results re-emphasize the early observations made in \textcite{castro2012generalized}.
Extending the reaggregated hull from the time-slot concept to the general precedence concept, the model in which the reaggregated hull reformulation outperformed both the Big-M and the hull reformulation of the same model. 
For the strip-packing problems, which \textcite{castro2012time} had identified a connection with the scheduling models, the reaggregated hull approach achieved the lowest mean solve times and smallest variances at every instance size, surpassing its hull counterpart and making it competitive with the Big-M formulations.
Using basic steps to build a formulation where the reaggregation technique would apply led to the surprising observation that symmetry-breaking effects, before reported as helpful for this type of problem, resulted in no apparent difference.
These findings confirm that collapsing disaggregated variables via shared-coefficient reaggregation preserves the convex-hull tightness while reducing model size, yielding both faster solves and more reliable convergence.

Future work will extend the evaluation of the reaggregated hull reformulation to a broader suite of benchmark problems.
We have identified it being a valid approach for model predictive control in the operation of a district energy system for PV self-consumption and grid decarbonization as presented in \cite{kim2022generalized}.
Moreover, formulations appearing in the representation of piece-wise linear constraints exhibit the shared linear coefficient structure that can be exploited in the reaggregated hull reformulation \cite{vielma2015mixed}.
We will investigate the performance of the reaggregation strategy to GDP models where the global constraints are nonlinear.
Examples include the surrogate-based decision-tree formulations recently proposed for engineering optimization applications in \cite{ammari2023linear}.

\section*{Supporting Information}

The following information is available in the Supporting Information file: Specifications of single-unit scheduling instances (Examples 1–10); 
model-size comparisons for BM, HR, and RHR formulations; 
extended solver results for Gurobi, SCIP, and HiGHS; 
performance profiles for single-unit scheduling and strip-packing benchmarks.

\section*{Acknowledgments}
The authors acknowledge and express sincere gratitude to the Davidson School of Chemical Engineering at Purdue University for providing a supportive research environment and funding.
We thank \texttt{Gurobi} Optimization for providing access to software through their academic licensing program.
Finally, D.E.B.N. would like to dedicate this work to the memory of Pedro Castro.
We are just building on top of his ideas in this manuscript.
His presence in conferences, his sense of humor, and his passion for everything he did (including solving optimization problems or watching soccer games) will be greatly missed.

\printbibliography

\newpage

\rule{0.05in}{1.75in}%
\begin{minipage}[b][1.75in]{3.25in}
    \centering
    \sffamily
    \frenchspacing

    \includegraphics[width=\linewidth]{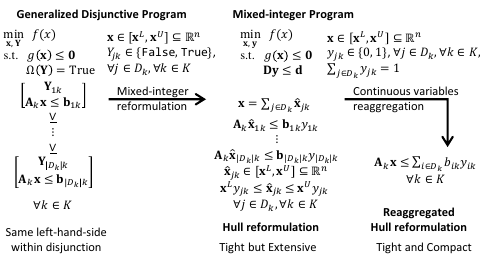}

\end{minipage}%
\rule{0.05in}{1.75in}

\end{document}